\documentclass{amsart}
\renewcommand{\baselinestretch}{1.5}
\usepackage{graphicx}
\usepackage{amsfonts}
\usepackage{mytab}
\usepackage{multicol}

\theoremstyle{plain}
\newtheorem{thm}{Donotwrite}[section]

\newtheorem{definition}[thm]{Definition}
\newtheorem{theorem}[thm]{Theorem}

\newtheorem{corollary}[thm]{Corollary}

\theoremstyle{definition}
\newtheorem{example}[thm]{Example}

\numberwithin{equation}{section}

\newcommand{\nc}{\newcommand}

\nc{\op}{\oplus} \nc{\pv}{P^{\vee}}

\nc{\B}{\mathbf{B}} \nc{\V}{\mathbf{V}} 
\nc{\nbinom}[2]{\genfrac{}{}{0pt}{1}{{#1}}{{#2}}}
\nc{\qbinom}[2]{\left[\genfrac{}{}{0pt}{1}{{#1}}{{#2}}\right]}
\nc{\ft}{\tilde{f}} 
\nc{\et}{\tilde{e}} 
\nc{\Y}{\mathbf{Y}}
\nc{\T}{\mathbf{T}}
\nc{\ra}{\rightarrow} 
\nc{\vep}{\varepsilon} 
\nc{\vp}{\varphi}
\nc{\h}{\mathfrak{h}} 
\nc{\oP}{\overline{P}}
\nc{\Fit}{\tilde{F}_i} 
\nc{\Eit}{\tilde{E}_i}
\nc{\fit}{\tilde{f}_i} 
\nc{\eit}{\tilde{e}_i}
\nc{\mf}{\mathfrak}
\nc{\ds}{\displaystyle}
\nc{\oa}{a'}
\nc{\ob}{b'}
\nc{\mc}{\mathcal}
\nc{\bsx}{\boldsymbol{x}}
\nc{\imin}{i_{min}}
\nc{\imax}{i_{max}}

\nc{\dmin}{d_{min}}
\nc{\dmax}{d_{max}}

\allowdisplaybreaks

\begin{document}

\title{Lattice Paths, Young Tableaux, and Weight Multiplicities}

\author{Rebecca L. Jayne}
\address{Hampden-Sydney College, Hampden-Sydney, VA 23943}
\email{rjayne@hsc.edu} 

\author{Kailash C. Misra}
\address{Department of Mathematics, North Carolina State University,  Raleigh,  NC 27695-8205}
\email{misra@ncsu.edu}

\thanks{KCM: partially supported by Simons Foundation grant \#  307555}

\begin{abstract}
For $\ell \geq 1$ and $k \geq 2$, we consider certain admissible sequences of $k-1$ lattice paths in a colored $\ell \times \ell$ square.  We show that the number of such admissible sequences of lattice paths is given by the sum of squares of the number of standard Young tableaux of partitions of $\ell$ with height $\leq k$, which is also the number of $(k+1)k\cdots21$-avoiding permutations of $\{1, 2, \ldots, \ell\}$.  Finally, we apply this result to the representation theory of the affine Lie algebra $\widehat{sl}(n)$ and show that this quantity gives the multiplicity of  certain maximal dominant weights in the irreducible module $V(k\Lambda_0)$.
\end{abstract}

\maketitle

\section{Introduction}\label{intro}
For fixed integers $\ell \geq 1$ and $k \geq 2$, we consider the admissible sequences of $k-1$ lattice paths in a colored $\ell \times \ell$ square given in \cite{JM}.  Each admissible sequence of paths can be associated with a partition $\lambda$ of $\ell$.  In Section \ref{paths}, we show that the number of self-conjugate admissible sequences of paths associated with $\lambda$ is equal to the number of standard Young tableaux of shape $\lambda$, and thus can be calculated using the hook length formula.  We extend this result to include the non-self-conjugate admissible sequences of paths and show that the number of all such admissible sequences of paths is equal to the sum of squares of the number of standard Young tableaux of partitions of $\ell$ with height less than or equal to $k$. Using the RSK correspondence in \cite{Sc},  it is shown in (\cite{Sta}, Corollary 7.23.12)  that the sum of squares of the number of standard Young tableaux of partitions of $\ell$ with height less than or equal to $k$ is equal to the number of $(k+1)k\cdots21$-avoiding permutations of $\{1, 2, \ldots, \ell\}$.

In Section \ref{multiplicities}, we apply our results to the representation theory of the affine Kac-Moody algebra $\widehat{sl}(n)$. Let 
$\{\alpha_0, \alpha_1, \ldots, \alpha_{n-1}\}$, $\{h_0, h_1, \ldots, h_{n-1}\},$ and $\{\Lambda_0, \Lambda_1, \ldots, \Lambda_{n-1}\}$ denote the simple roots, simple coroots, and fundamental weights respectively. Note that $\Lambda_j(h_i) = \delta_{ij}$. For $1 \leq \ell \leq \left \lfloor \frac{n}{2} \right \rfloor$, set $\gamma_\ell = \ell \alpha_0 + (\ell - 1)\alpha_{1} + (\ell - 2) \alpha_{2} + \cdots + \alpha_{\ell- 1} + \alpha_{n-\ell+1} + \cdots +  (\ell - 2)\alpha_{n-2} + (\ell - 1)\alpha_{n-1}$ and $\mu_\ell = k\Lambda_0 - \gamma_\ell$. As shown in \cite{JM}, $\mu_\ell$ are maximal dominant weights of the irreducible $\widehat{sl}(n)$-module $V(k\Lambda_0)$. We show that the multiplicity of the weight $\mu_\ell$ in $V(k\Lambda_0)$ is the number of $(k+1)k\cdots21$-avoiding permutations of $\{1, 2, \ldots, \ell\}$, which proves Conjecture 4.13 in \cite{JM}.

\section{Lattice Paths} \label{paths}

For fixed integers $\ell \geq 1$ and $k \geq 2$, consider the $\ell \times \ell$ square containing $\ell ^2$ unit boxes  in the fourth quadrant so that the top left corner of the square is at the origin. We assign color $a+b$ to a box if its upper left corner has coordinates $(a,b)$. This gives the following 
$\ell \times \ell$ colored square $Y$:
 
\begin{figure}[h]
\centering
  $$ \normalsize \tableau{0 & 1 &  \cdots & \ell - 2 & \ell -1 \\ -1&  0 & \cdots & \ell - 3 & \ell -2\\  \vdots  & \vdots &  \ddots & \vdots & \vdots \\ 2-\ell  &  3-\ell  & \cdots & 0 & 1\\ 1- \ell & 2 - \ell & \ldots & -1 & 0 }$$
 \caption{$Y$}
\label{thesquare}
 \end{figure}
 
A lattice path $p$ on $Y$ is a path joining the lower left corner $(0, -\ell)$ to the upper right corner $(\ell, 0)$ moving unit lengths  up or right. For two lattice paths $p, q$ on $Y$ we say that $p \leq q$ if the boxes above $q$ are also above $p$.  Now, we draw $k-1$ lattice paths, $p_1, p_2, \ldots, p_{k-1},$ on $Y$ such that $p_1 \leq p_2 \leq \cdots \leq p_{k-1}$. For integers $i, j$, where $2 \leq i \leq k - 1$, $-(\ell-1) \leq j \leq \ell-1$, we define $t_i^j \geq 0$ to be the number of $j$-colored boxes between $p_{i-1}$  and $p_{i}$.  We define $t_1^j$ to be the number of $j$-colored boxes below $p_1$ and $t_0^j$  to be the number of $j$-colored boxes above $p_{k-1}$.  
 
\begin{definition}   \label{pathsdef} \cite{JM} We call a sequence of lattice paths $p_1\leq p_2 \leq \ldots \leq p_{k-1}$ on $Y$ admissible if it satisfies the following conditions:  
\begin{enumerate}

\item \label{conds}  $p_1$ does not cross the diagonal $y = x - \ell$  and 
\item for  $2 \leq i \leq k-1$, $-(\ell-1) \leq j \leq \ell-1$, we have \\

 $t_i^j \leq \min \left \{t_{i-1}^j, \ell - |j| - t_1^j - \ds \sum_{a=1}^{i-1}t_a^j \right \}$ and $\begin{cases}
t_i^j \leq t_i^{j-1}, \quad j>0 \\
t_i^j \leq t_i^{j+1}, \quad j <0.
\end{cases}$

\end{enumerate}
\end{definition}
Denote by $\mc{T}_\ell^k$ the set of all admissible sequences of $k-1$ paths.  Notice that there are $\ell$ 0-colored boxes in $Y$ and hence for any admissible sequence of paths, $\ds \sum_{i=0}^{k-1}t_i^0 = \ell$.  In addition, it follows from Definition \ref{pathsdef}(2) that $t_0^0 \geq t_1^0 \geq \cdots \geq t_{k-1}^0$ for any admissible sequence of paths.  Thus, we can and do associate an admissible sequence of paths $p_1\leq p_2 \leq \ldots \leq p_{k-1}$ on $Y$ with a partition $\lambda = (t_0^0, t_1^0, \ldots, t_{k-1}^0)$ of $\ell$.  In this case, we say that this admissible sequence of paths is of type $\lambda$ and often draw $\lambda$ as a Young diagram.  

\begin{example}  Figure \ref{AdSeq}(a) is an element of $\mc{T}_6^4$, where $p_1, p_2,$ and $p_3$ are shown in Figures \ref{AdSeq}(b), \ref{AdSeq}(c), and \ref{AdSeq}(d), respectively.  Notice that this admissible sequence of paths is of type $\lambda = (2,2,1,1)$.
\begin{figure*}[!h]
\centering
\begin{tabular}{cc}
\includegraphics[scale=.7]{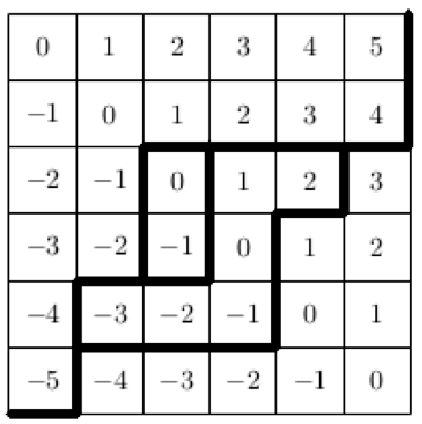} & \includegraphics[scale=.7]{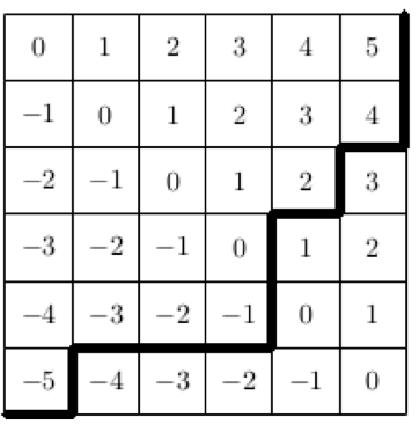} \\ (a) & (b) \\
 \includegraphics[scale=.7]{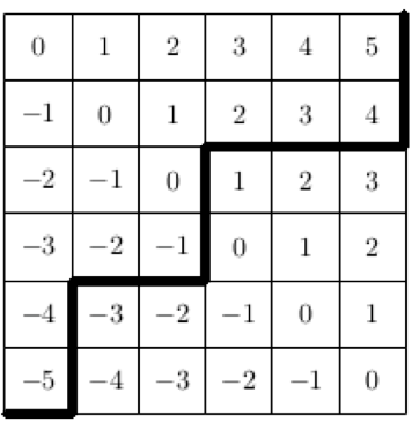} & \includegraphics[scale=.7]{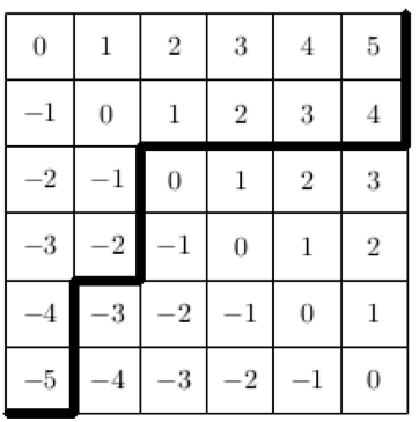} \\ 
(c) & (d) \\ 
\end{tabular}
\caption{Admissible sequence of paths}
 \label{AdSeq}
\end{figure*}
\end{example}

For a given partition $\lambda = (\lambda_1, \lambda_2, \ldots, \lambda_k) \vdash \ell$, we let $f^\lambda$ denote the number of standard Young tableaux of shape $\lambda$ and let $\lambda' = (\lambda_1', \lambda_2', \ldots, \lambda_m')$ be its conjugate partition.  For a Young diagram of type $\lambda$, the hook length at a box $u$ in row $i$ and column $j$  is given by $h(u) = \lambda_i + \lambda_j' - i - j + 1$.  The hook length formula states that $\ds f^\lambda = \frac{\ell!}{\ds \Pi_{u \in \lambda}h(u)}$.  We denote the height of a standard Young tableau of shape $\lambda$ by $ht(\lambda)$.  

\begin{example}
For the partition $\lambda = (4,2,1) \vdash 7$, we have the Young diagram in Figure \ref{yd}(a).  Using the hook length formula, and moving across the rows and then down in the diagram, we calculate that the number of standard Young tableaux of shape $\lambda$ is $f^{\lambda} = \ds \frac{7!}{6 \cdot 4 \cdot 2 \cdot 1 \cdot 3 \cdot 1 \cdot 1} = 35.$ Figure \ref{yd}(b) shows one such standard Young tableau of shape $\lambda$.  

\begin{figure*}[h]
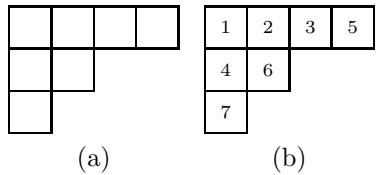

\centering
\begin{tabular}{cc}
{\scriptsize \tableau{  & & & \\  & \\ \\ }} & {\scriptsize \tableau{ 1 & 2& 3& 5\\ 4 &6 \\7 \\ }} \\ 

(a) & (b) \\ 
\end{tabular}
\caption{Young diagram and standard Young tableau of shape $(4,2,1)$}
 \label{yd}
\end{figure*}
\end{example}

A lattice path on $Y$ is self-conjugate if it is symmetric about the diagonal from $(0,0)$ to $(\ell, -\ell)$. A sequence of lattice paths is self-conjugate if all lattice paths in the sequence are self-conjugate.  Next, we give a relationship between self-conjugate admissible sequences of paths and standard Young tableaux.  We denote the set of all self-conjugate admissible sequences of $k-1$  paths on the $\ell \times \ell$ colored square $Y$ by $S_\ell^k$.  Let $F_\ell^k$ denote the set of all standard Young tableaux with $\ell$ boxes and height $\leq k$.  Then $|F_\ell^k| =  \ds  \sum_{\lambda \vdash \ell, h(\lambda) \leq k} f^\lambda$.  

Now, we define a map $\tau \colon F_\ell^k \to S_\ell^k$ as follows.  Let $X \in F_\ell^k$.  We draw the first $\ell$ moves of each path, $p_1, p_2, \ldots, p_{k-1},$ on the $\ell \times \ell$ colored square $Y$, giving the portion of each path from $(0,-\ell)$ to the diagonal connecting $(0,0)$ to $(\ell,-\ell)$, move by move.  In the $i^{th}$ path $p_i$, the $j^{th}$ move is up when $j$ is in row 2 through row $(i+1)$ and the $j^{th}$ move is right otherwise.  Once the first $\ell$ moves of every path are drawn, we complete each path by reflecting over the diagonal, obtaining a self-conjugate sequence of paths.  

Although the image $\tau (X)$ is a self-conjugate sequence of $k-1$ paths by definition, we still need to show that it is an admissible sequence of paths.  Notice that for path $p_1$ in $\tau(X)$, the $m^{th}$ up move, move $v$, is represented in column $m$ of row 2 of $X$.  Thus, there are at least $m$ entries in row 1 of $X$ less than or equal to $v$, which correspond to $m$ right moves prior to the $m^{th}$ up move in $p_1$ of $\tau(X)$.  Hence, $p_1$ does not cross the diagonal $y = x - \ell$.  For $1-\ell \leq j \leq 0$, we call the line $y = -x+j$ connecting $(0,j)$ and $(\ell+j, -\ell)$ the $j$-diagonal passing through all $j$-colored boxes. Note that it takes $\ell + j$ moves to go from $(0, - \ell)$ to the $j$-diagonal. By definition, $t_0^j$ is the number of right moves in the path $p_{k-1}$ below the $j$-diagonal, which equals the number of entries in the first row of $X$ that are less than or equal to $\ell +j$. For $2 \leq i \leq k-1$, the number of up moves in the path $p_i$ below the $j$-diagonal is equal to the number of $j$-colored boxes below $p_i$. Hence, by definition, $t_i^j$ is the number of up moves below the $j$-diagonal that are in $p_i$ but not in $p_{i-1}$. Therefore, the number of up moves in $p_i$ (resp. $p_{i-1}$) below the $j$ diagonal is $t_1^j + t_2^j + \cdots + t_i^j$ (resp. $t_1^j + t_2^j + \cdots + t_{i-1}^j$). Since the paths are self-conjugate, using symmetry it follows that  $p_i \geq p_{i-1}$. Hence $p_1 \leq p_2 \leq \cdots \leq p_{k-1}$.  Now observe that by definition of $\tau$, $t_i^j$ equals the number of boxes in row $i +1$ of $X$ with entries less than or equal to $\ell+j$. Since $X$ is a standard tableau, we have $t_i^j \leq t_{i-1}^j$.  There are at least $t_1^j$ boxes in the second row of $X$ and hence in the first row of $X$. So it follows that $2t_1^j + t_2^j + \cdots + t_i^j \leq \ell +j$ which implies $t_i^j \leq \ell + j - t_1^j - \ds \sum_{a=1}^{i-1}t_a^j$. Hence using the symmetry of the paths we have $t_i^j \leq \min \left \{t_{i-1}^j, \ell - |j| - t_1^j - \ds \sum_{a=1}^{i-1}t_a^j \right \}$. Finally, for $j < 0$, $t_i^j$ (resp. $t_i^{j+1}$) is the number of boxes in row $i+1$ of $X$ with entries less than or equal to $\ell + j$ (resp. $\ell + j +1$). So $t_i^j \leq t_i^{j+1}$ for $ j < 0$. Now using symmetry and reflecting on the diagonal $y = -x$, we have $t_i^j \leq t_i^{j-1}$ for $ j > 0,$ proving that $\tau (X) \in S_\ell^k$.

\begin{theorem}\label{tau} For $k \geq 2 , \ell \geq 1,$ the map $\tau \colon F_\ell^k \to S_\ell^k$  is a bijection. Hence $|S_\ell^k| =  \ds  \sum_{\lambda \vdash \ell, h(\lambda) \leq k} f^\lambda$. In particular, the number of self-conjugate admissible sequences of paths $p_1 \leq p_2 \leq \cdots \leq p_{k-1}$ of type $\lambda = (t_0^0, t_1^0, \cdots , t_{k-1}^0) \vdash \ell$ is equal to the number of standard Young tableaux of shape $\lambda$. 
\end{theorem}
\begin{proof} To prove the statement, it is sufficient to show that the map $\tau$ has an inverse. We define the map $\sigma \colon S_\ell^k \to F_\ell^k$ as follows. Let $Z = \{p_1 \leq p_2 \leq \cdots \leq p_{k-1}\} \in S_\ell^k$ be a self-conjugate admissible sequence of paths on the $\ell \times \ell$ square $Y$ of type $\lambda = (t_0^0, t_1^0, \ldots, t_{k-1}^0) \vdash \ell$. Note that $t_m^0$ equals the number of up moves in the path $p_m$ but not in any path $p_i , 1 \leq i < m$. We draw the Young diagram $X$ of shape $\lambda$. Now we fill in the boxes of $X$ with numbers $1, 2, \ldots , \ell$ as follows.  We traverse, in order, the first $\ell$ moves from $(0, -\ell)$ to the $0$-diagonal in each path $p_1, p_2, \ldots, p_{k-1}$. If the $i^{th}$ move is a right move for all paths, then we place the number $i$ in the leftmost available box in the first row of $X$. If the $i^{th}$ move is an up move in any path, then choose $m$ to be the smallest integer such that the up move occurs in the path $p_m$. In such case, we place the number $i$ in the leftmost available box of row $m+1$ of $X$.   

By construction, the entries in the rows of $X$ increase from left to right.  Notice that for $2 \leq i \leq k-1$, any up moves in the path $p_i$ that do not occur in $p_{i-1}$ are represented in row $i+1$ of $X = \sigma(Z)$. So for $-(\ell -1) \leq j \leq 0$, $t_i^j$  in $Z$ is equal to the number of entries in row $(i+1)$ of $X$ that are less than or equal to $\ell+j$. To see that the entries in the columns increase from top to bottom, suppose for the sake of contradiction that for some row $r$ ($2 \leq r \leq k$) the entry in column $c$, say $a$, is less than the entry in column $c$ of row $r-1$.  Since the entries in $X$ increase from left to right, the number of entries in row $r$ less than or equal to $a$ is $c$.  Hence $t_{r-1}^{a - \ell} = c.$ Since $t_{r-2}^{a - \ell}$ is the number of entries in row $r-1$ less than or equal to $a$ and the entry in column $c$ of row $r-1$ is greater than $a$, we have $t_{r-2}^{a - \ell} < t_{r-1}^{a - \ell} = c$, which contradicts Definition \ref{pathsdef}.  Hence $X$ is a standard Young tableaux and $X \in F_\ell^k$.  Using the symmetry of the paths and the definition of $\tau$ it follows that $\tau (\sigma(Z)) = Z$, which proves the theorem.
\end{proof}

\begin{example}
Notice that Figure \ref{mapextab}(a) is an element of $F_6^4$.  To apply $\tau$, we draw three paths to the diagonal, with moves shown in Figure \ref{mapextab}(b).  Figures \ref{mapextab}(c), \ref{mapextab}(d), and \ref{mapextab}(e) show the first $6$ moves of $p_1$, $p_1$ and $p_2$, and $p_1, p_2$ and $p_3$, respectively.  These paths are reflected over the diagonal to form an element of $S_6^4$ as shown in Figure \ref{mapextab}(f).

\begin{figure}[h]
\begin{centering}

\begin{multicols}{2}

\ \includegraphics[scale=.7]{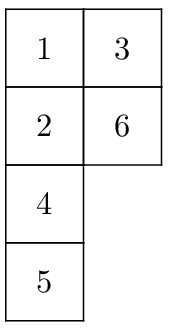} 

(a)

\begin{tabular}{lllllll} 
\hline\noalign{\smallskip}
 & \multicolumn{6}{c} {Move} \\
 \noalign{\smallskip}\hline\noalign{\smallskip}
Path & 1 & 2 & 3 & 4 & 5 & 6 \\ 
\noalign{\smallskip}\hline\noalign{\smallskip}
1 & R & U & R & R & R & U \\ 
2 & R & U & R & U & R & U \\
3 & R & U & R & U & U & U \\ 
\noalign{\smallskip}\hline
\end{tabular} 

\smallskip

(b)

\end{multicols}

\begin{tabular}{ccc}
\includegraphics[scale=.75]{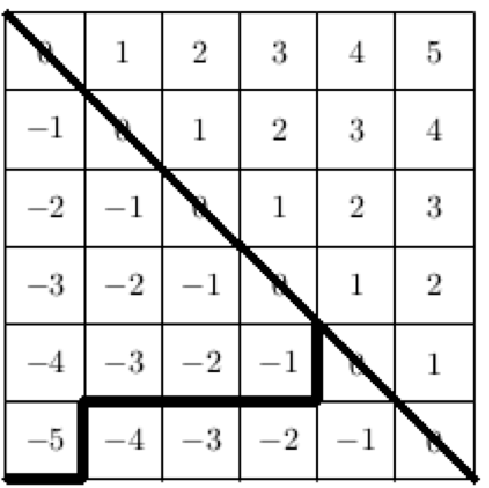} & \hspace{.5in} &  \includegraphics[scale=.75]{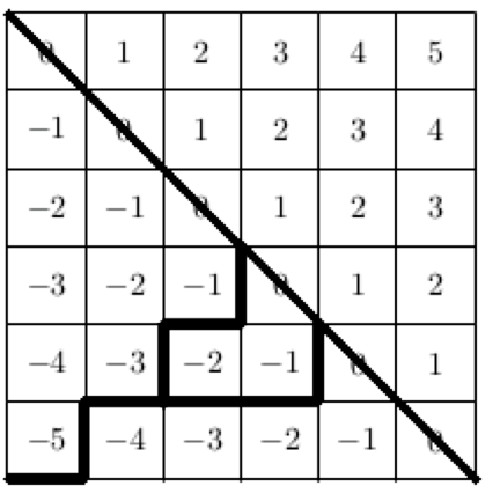} \\
(c) & &  (d) \\
\includegraphics[scale=.75]{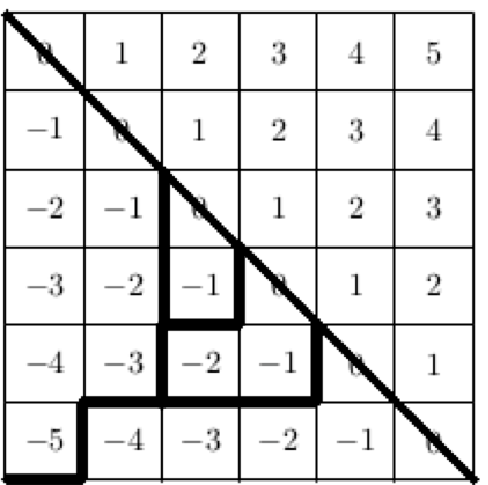} & & \includegraphics[scale=.75]{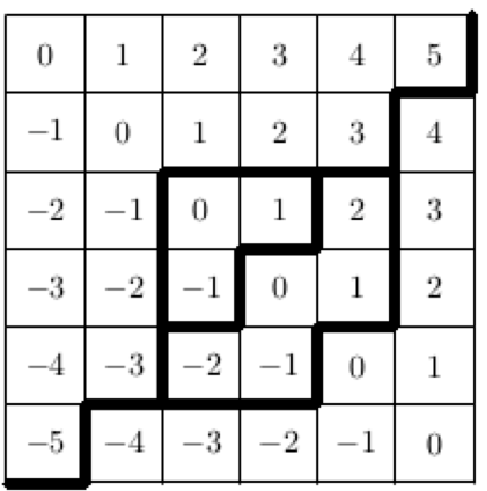} \\
(e) & &  (f) \\ 
\end{tabular}

\end{centering}
\caption{Example of $\tau$ map}
\label{mapextab}
\end{figure}
\end{example}

\begin{example}
Let $Z \in S_6^4$ as in Figure \ref{sigmamap}(a).  In Figure \ref{sigmamap}(b), we record the moves for each path.  We construct $X \in T_6^4$ as in Figure \ref{sigmamap}(c) by first drawing the Young diagram of shape $\lambda = (t_0^0 = 2, t_1^0 = 2, t_2^0 = 1,t_3^0 =1)$.  We fill in the integers 1 through $6$ using the information collected in Figure \ref{sigmamap}(b) and applying $\sigma$.  For example, since move 1 for all paths is right, we place 1 in row 1 of $X$.  Since move 6 has its first up move in $p_3$, we place 6 in row 4.  

\begin{figure*}[h]

\begin{centering}

\begin{multicols}{2}
\includegraphics[scale=.75]{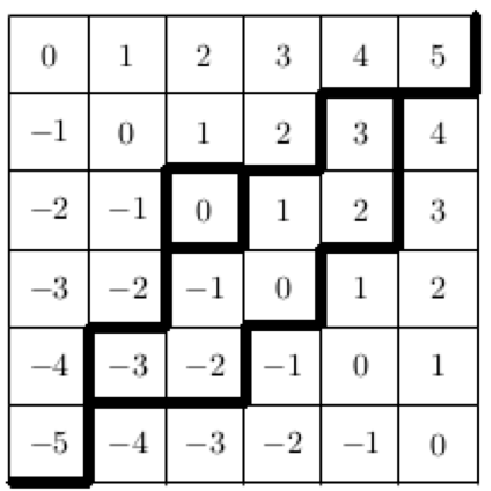} 

\begin{tabular}{llllllll} \hline \noalign{\smallskip}
 & \multicolumn{6}{c} {Move} \\ 
 \noalign{\smallskip}\hline\noalign{\smallskip}
Path & 1 & 2 & 3 & 4 & 5 & 6 \\ 
\noalign{\smallskip}\hline\noalign{\smallskip}
1 & R & U & R & R & U & R \\ 
2 & R & U & U & R & U & R \\ 
3 & R & U & U & R & U & U \\
\noalign{\smallskip}\hline
\end{tabular}

\end{multicols}

\begin{multicols}{2}
(a)

(b)
\end{multicols}

\includegraphics[scale=.75]{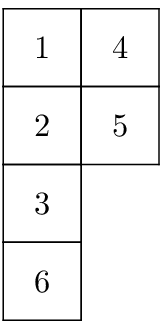}

(c)

\end{centering}
\caption{Example of $\sigma$ map}
\label{sigmamap}
\end{figure*}
\end{example}
  
\begin{theorem} \label{squaring} For $k \geq 2$ and $\ell \geq 1$, let  $\lambda = (t_0^0, t_1^0, \ldots, t_{k-1}^0) \vdash \ell$.  Then the number of admissible sequences of $k-1$ paths on a colored $\ell \times \ell$ square $Y$ of type $\lambda = (t_0^0, t_1^0, \ldots, t_{k-1}^0)$ is 
$(f^\lambda)^2$.  Hence $|\mc{T}_\ell^k| =  \ds  \sum_{\lambda \vdash \ell, h(\lambda) \leq k} (f^\lambda)^2$.
\end{theorem}

\begin{proof}  Let $Z_1 = \{p_1 \leq p_2 \leq \cdots \leq p_{k-1}\}$ and $Z_2 = \{p'_1 \leq p'_2 \leq \cdots \leq p'_{k-1}\}$ be two self-conjugate admissible sequences of  paths on $Y$ of type $\lambda = (t_0^0, t_1^0, \ldots, t_{k-1}^0) \vdash \ell$. Then for $1 \leq i \leq k-1$, combining the portion of the path $p_i$ below the $0$-diagonal on $Y$ with the portion of the path $p'_i$ above the $0$-diagonal, we obtain an admissible sequence of $k-1$ paths of type $\lambda$ which is not self-conjugate unless $Z_1 = Z_2$. Conversely, if we have an admissible sequence of paths $Z = \{p_1 \leq p_2 \leq \cdots \leq p_{k-1}\}$ of type $\lambda\vdash \ell$, we can obtain a pair of self-conjugate admissible sequences of paths $(Z_1 , Z_2)$ each of type $\lambda,$ where $Z_1$ (resp. $Z_2$) is obtained by reflecting the portions of the paths in $Z$ below (resp. above) the $0$-diagonal of $Y$. Thus the set of admissible sequences of $k-1$ paths of type $\lambda\vdash \ell$ is in bijection with the set of pairs of self-conjugate admissible sequences of $k-1$  paths of type $\lambda\vdash \ell$. Hence by Theorem \ref{tau}, the number of admissible sequences of $k-1$ paths in $Y$ of type $\lambda\vdash \ell$ is $(f^\lambda)^2$. Since any partition $\lambda \vdash \ell$ of height less than or equal to $k$ is associated with an admissible sequence of $k-1$ paths of type $\lambda$, we have $|\mc{T}_\ell^k| =  \ds  \sum_{\lambda \vdash \ell, h(\lambda) \leq k} (f^\lambda)^2$.
\end{proof}

\begin{example}
In Figure \ref{selffig}(a), we have $Z = \{ p_1 \leq p_2 \leq p_3 \}$, a non-self-conjugate admissible sequence of paths of type $\lambda = (2,2,1,1)$.  Reflecting the portion of $Z$ below the $0$-diagonal gives the self-conjugate sequence $Z_1$ in Figure \ref{selffig}(b) and reflecting the portion of $Z$ above the $0$-diagonal gives the self-conjugate sequence $Z_2$ in Figure \ref{selffig}(c), both of type $\lambda$. Thus, by the bijection in the proof of Theorem \ref{squaring}, the non-self-conjugate sequence of paths $Z$ corresponds to the pair of self-conjugate sequence of paths $(Z_1 , Z_2)$.

\begin{figure*}[h]

\begin{tabular}{ccc}
\includegraphics[scale=.75]{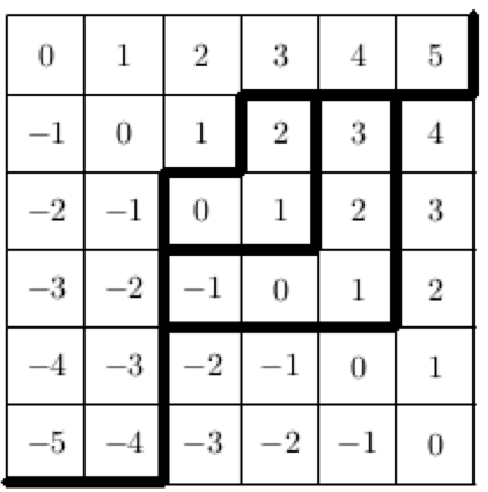} &   \includegraphics[scale=.75]{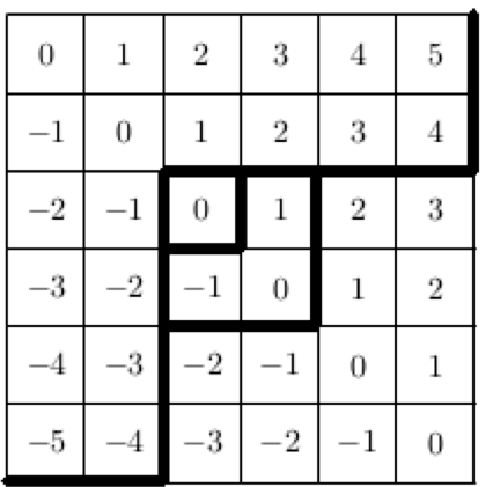} & \includegraphics[scale=.75]{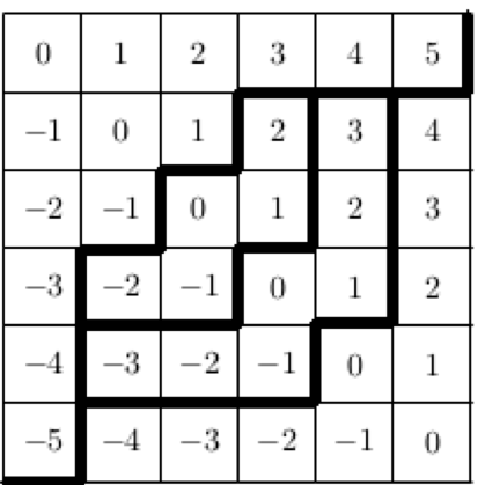}    \\ 
(a) & (b) & (c) \\ 
\end{tabular}
\caption{Non-self-conjugate admissible sequence of paths correspondence with pair of self-conjugate admissible sequences of paths}
\label{selffig}
\end{figure*}
\end{example}


\section{Weight Multiplicities}\label{multiplicities}

In this section we use the result  in Theorem \ref{squaring} to prove  Conjecture 4.13 in \cite{JM}, giving a formula for certain weight multiplicities of the affine Lie algebra $\mathfrak{g}=\widehat{sl}(n) = sl(n) \otimes \mathbb{C}[t, t^{-1}] \oplus \mathbb{C} c \oplus \mathbb{C}d, n\geq 2,$ where $sl(n)$ is the simple Lie algebra of $n\times n$ trace zero matrices, $c$ is the canonical central element and $d= 1 \otimes \frac{d}{dt}$ is a degree derivation. We recall some facts about this affine Lie algebra from \cite{Kac}.

The generalized Cartan matrix associated with the affine Lie algebra $\mathfrak{g}$ is the $n \times n$ matrix $A = (a_{ij})_{0 \leq i,j \leq n-1}$ where $a_{ii} = 2, a_{ij} = -1$  for $ |i-j| = 1, a_{0,n-1} = a_{n-1,0} = -1,$ and $a_{ij} = 0$ otherwise. Let $\Pi = \{ \alpha_0, \alpha_1, \ldots, \alpha_{n-1}\}$ and $\Pi^{\vee} = \{h_0, h_1, \ldots, h_{n-1}\}$ denote the simple roots and simple coroots respectively. We recall that $\alpha_j(h_i) = a_{ij}$ and $\delta = \alpha_0 + \alpha_1 + \cdots +  \alpha_{n-1}$ is the null root. The Cartan subalgebra of $\mf{g}$ is $\mf{h} = $span$_{\mathbb{C}} \{h_1, h_2, \ldots, h_{n-1},d\}$. The fundamental weights $\{\Lambda_0 , \Lambda_1 , \ldots, \Lambda_{n-1}\}$ are defined by $\Lambda_{i}(h_{j}) = \delta_{ij}, \Lambda_{i}(d) = 0$ and the set $P = \mathbb{Z} \Lambda_{0} \oplus \mathbb{Z} \Lambda_{1} \oplus \cdots \oplus \mathbb{Z} \Lambda_{n-1} \oplus \mathbb{Z}\delta$ (resp.  $P^\vee = \mathbb{Z}h_{0} \oplus \mathbb{Z} h_{1} \oplus \cdots \oplus \mathbb{Z} h_{n-1} \oplus \mathbb{Z}d)$ is called the weight lattice (resp. coweight lattice).

The set  $P^+ =  \{ \Lambda \in P \mid \Lambda({h_i}) \geq 0, \forall  i \}$ is called the set of dominant integral weights.  For each $\Lambda \in P^+$, there is an irreducible highest weight $\mathfrak{g}$-module $V(\Lambda)$ with highest weight $\Lambda$.  For $\mu \in \mf{h}^*$, if  $V(\Lambda)_{\mu} = \{v \in V(\Lambda) \mid h(v) = \mu (h)v,    \forall h \in \mathfrak{h} \} \not = \{0\}$, then $\mu$ is called a weight of $V(\Lambda)$.    The multiplicity of $\mu$ in $V(\Lambda)$, denoted by $mult_\Lambda(\mu)$, is the dimension of $V(\Lambda)_\mu$.  If $\mu$ is a weight of $V(\Lambda)$ and $\mu + \delta$ is not a weight, then $\mu$ is called a maximal weight.   We denote the set of all maximal weights of $V(\Lambda)$ by max$(\Lambda)$. Hence, max$(\Lambda) \cap P^+$ is the set of all maximal dominant weights of $V(\Lambda)$.  

For $1 \leq \ell \leq \left \lfloor \frac{n}{2} \right \rfloor$, let $\gamma_\ell = \ell \alpha_0 + (\ell - 1)\alpha_{1} + (\ell - 2) \alpha_{2} + \cdots + \alpha_{\ell- 1} + \alpha_{n-\ell+1} + \cdots +  (\ell - 2)\alpha_{n-2} + (\ell - 1)\alpha_{n-1}$.  
It was shown in \cite{Tsu} that $ \max(2\Lambda_0) \cap P^+ = \left \{ 2\Lambda_0 - \gamma_\ell \mid 1 \leq \ell \leq \left \lfloor \frac{n}{2} \right \rfloor  \right \}$. Indeed,  $\left \{ k\Lambda_0 - \gamma_\ell \mid 1 \leq \ell \leq \left \lfloor \frac{n}{2} \right \rfloor  \right \} \subseteq \max(k\Lambda_0) \cap P^+$  for $k \geq 2$ (see \cite{JM}). The following Theorem was proved in \cite{JM}.  

\begin{theorem}\label{PathsTh} \cite{JM}  For $k \geq 2$, $1 \leq \ell \leq \left \lfloor \frac{n}{2} \right \rfloor$, we have $mult_{k\Lambda_0}(k\Lambda_0 - \gamma_\ell) = |\mathcal{T}_{\ell}^{k}|$.
\end{theorem}

Hence  the following result  is a consequence of Theorem \ref{squaring} and Theorem \ref{PathsTh}.  

\begin{corollary} \label{mult_formula}
For $k \geq 2$, $1 \leq \ell \leq \left \lfloor \frac{n}{2} \right \rfloor$, we have $mult_{k\Lambda_0}(k\Lambda_0 - \gamma_\ell) = \ds \sum_{\lambda \vdash \ell, h(\lambda) \leq k} (f^\lambda)^2$.
   
\end{corollary}

We denote a permutation of $\{1,2,\ldots, \ell\}$ by a sequence $w=w_1w_2\ldots w_\ell$ if $ 1 \mapsto w_1, 2 \mapsto w_2, \ldots , \ell \mapsto w_\ell$.  If a decreasing subsequence of $w$ has $m$ terms, we say it is a decreasing subsequence of length $m$.  A $(k+1)k\ldots21$-avoiding permutation is a permutation $w$ with no decreasing subsequence of length $k+1$.  For example, $w=26873415$ has longest decreasing subsequence of length 4 (namely, 8731 and 8741) and so $w$ is a 54321-avoiding permutation.  We have the following result as a consequence of (\cite{Sta}, Corollary 7.23.12)  and Theorem \ref{PathsTh} which proves Conjecture $4.13$ in \cite{JM}.

\begin{corollary} 
For $k \geq 2$, $1 \leq \ell \leq \left \lfloor \frac{n}{2} \right \rfloor$, we have $mult_{k\Lambda_0}(k\Lambda_0 - \gamma_\ell) =$  the number of $(k+1)k\cdots21$-avoiding permutations of $\ell$.   
\end{corollary}

\begin{example}
Consider maximal dominant weight $4\Lambda_0 - \gamma_5$ of $V(4\Lambda_0)$.  By Corollary \ref{mult_formula}, the multiplicity of $4\Lambda_0 - \gamma_5$ is equal to 
$\ds \sum_{\lambda \vdash 5, h(\lambda) \leq 4} (f^\lambda)^2$.  The partitions of 5 with height less than or equal to 4 are (5), (4,1), (3,2), (3,1,1),(2,2,1), and (2,1,1,1).  In Table \ref{mult_ex}, we calculate $f^\lambda$ for each of these partitions using the hook length formula.  Thus, $mult_{4\Lambda_0}(4\Lambda_0 - \gamma_5) = 1^2 + 4^2 + 5^2 + 6^2 + 5^2 + 4^2 = 119$.  

\begin{table}[h]
\caption{ Values of $f^\lambda$ for $\lambda \vdash 5$, $ht(\lambda) \leq 4$}
\label{mult_ex}
\centering
\begin{tabular}{|c|c|} \hline
$\lambda$ & $f^\lambda$ \\ \hline
& \\
{\scriptsize \tableau{ & & & & } }  & $\ds \frac{5!}{5 \cdot 4 \cdot 3 \cdot 2 \cdot 1} = 1$ \\ 
& \\ \hline & \\
{\scriptsize \tableau{ & & & \\ \\ } }  & $\ds \frac{5!}{5 \cdot 3 \cdot 2 \cdot 1 \cdot 1} = 4$ \\
& \\ \hline & \\
{\scriptsize \tableau{ & & \\ &  } }  & $\ds \frac{5!}{4 \cdot 3 \cdot 1 \cdot 2 \cdot 1} = 5$ \\
& \\ \hline & \\
{\scriptsize \tableau{ & & \\  \\ \\  } }  & $\ds \frac{5!}{5 \cdot 2 \cdot 1 \cdot 2 \cdot 1} = 6$ \\
& \\ \hline & \\
{\scriptsize \tableau{ &  \\ & \\  \\ } }  & $\ds \frac{5!}{4 \cdot 2 \cdot 3 \cdot 1 \cdot 1} = 5$ \\
& \\ \hline & \\
{\scriptsize \tableau{ &  \\  \\  \\  \\ } }  & $\ds \frac{5!}{5 \cdot 1 \cdot 3 \cdot 2 \cdot 1} = 4$ \\
& \\ \hline
\end{tabular}
\end{table}
\end{example}

\newpage


\end{document}